\documentclass[12pt,a4paper]{amsart}
\usepackage{geometry, empheq}
\usepackage[english]{babel}
\usepackage{amsmath}
\usepackage[textsize=scriptsize]{todonotes}
\usepackage{amssymb}
\usepackage{amsthm}
\usepackage{xcolor,graphicx}
\usepackage{color}
\usepackage{nicefrac}
\usepackage{caption}
\usepackage{subcaption}
\usepackage{float}
\usepackage[normalem]{ulem}
\usepackage{multirow}
\usepackage{enumerate}
\usepackage{hyperref}
\usepackage{tikz}

\usepackage[utf8]{inputenc}
\usepackage{pgfplots,pgfplotstable}
\usepackage{pgfgantt}
\usepackage{pdflscape}
\usepgfplotslibrary{external}
\usetikzlibrary{external}
\pgfplotsset{compat=newest}
\pgfplotsset{plot coordinates/math parser=false}
\usepackage{algorithm}
\usepackage{algorithmic}
\usepackage{calc}

\newcommand{\arrow}{\rightarrow}
\newcommand{\ds}{\displaystyle}

\newcommand{\kw}{\rule{2mm}{2mm}}

\DeclareMathOperator*{\prox}{prox}

\newcommand{\norm}[1]{{\|  #1\|}}
\renewcommand{\l}{\left}
\renewcommand{\r}{\right}
\newcommand{\om}{\Omega}

\newcommand{\reals}{\mathbb{R}}
\newcommand{\sign}{\mathrm{sign}}
\DeclareMathOperator*{\argmin}{argmin} 

\renewenvironment{proof}{{Proof.}}{\hfill\kw}

\newcounter{theassumption}
\newtheorem{assumption}[theassumption]{{Assumption}}

\newtheorem{definition}{{Definition}}

\newtheorem{remark}{{Remark}}

\newtheorem{theorem}{{Theorem}}

\begin{document}
\title[A second-order method for composite sparse optimization problems]{An enriched second-order method for nonconvex composite sparse optimization problems }

\author{J.C. De los Reyes$^\ddag$ and P. Merino$^\ddag$}
\address{$^\ddag$Research Center on Mathematical Modeling (MODEMAT) and Department of Mathematics, Escuela Polit\'ecnica Nacional, Quito, Ecuador}
\thanks{$^*$This research has been supported by project PIMI-17-01 granted by Escuela Politécnica Nacional, Quito--Ecuador.}

\smallskip
\begin{abstract}
In this paper we propose a second--order method for solving composite sparse optimization problems consisting of minimizing the sum of a differentiable (possibly nonconvex) function and a nondifferentiable convex term. The composite nondifferentiable convex penalizer is given by the $1$--norm of a matrix multiplied with the coefficient vector. The proposed algorithm relies on the three main ingredients: the minimum norm subgradient, a projection step and generalized second--order information associated to the nondifferentiable term. By combining these ideas, we devise a generalized second--order method for solving composite sparse optimization problems, for which the convergence analysis is carried out. Problems involving the minimization of the \emph{anisotropic total variation} or \emph{differential graph operators} can be efficiently solved with the proposed algorithm. We present several computational experiments to show the performance of our approach for different application examples.
\end{abstract}
\maketitle

\section{Introduction}

The composite problem of minimizing the cost $f(x) + \beta \norm{Cx}_1$, with $f$ differentiable and for some matrix $C$, is relevant in practice when sparsity is affected by a given pattern matrix. For example, when $C$ corresponds to the successive difference operator, then $\norm{C x}_1$ becomes the anisotropic total variation of $x$, which has several applications in signal and image processing \cite{tibshirani05}.
Moreover, higher order differential operators (e.g., the graph Laplacian) may be covered by $C$, which arise in, e.g., trend filtering over graphs \cite{wang2016trend} or nonlocal image denoising \cite{gilboa2009nonlocal}.

While first-order algorithms have been extensively developed for minimizing special cases of the objective function $f(x) + \beta \norm{Cx}_1$, mainly with $f$ strictly convex (see, e.g., \cite{beck2009,chambollepock2016}), second--order methods have not really been focus of attention in the context of nonsmooth optimization, despite their well-known superlinear convergence properties. One of the reasons for the lack of popularity is related to the high storage requirements and computational cost at each iteration, that turn out to prohibitive in absence of additional computing tools. However, second--order methods can be practical and advantageous if combined with cost--reduction and parallelization techniques \cite{anil2020}. Moreover, differently from most first-order methods, they are well-suited for handling nonconvex costs, which are incresingly important for image processing tasks, e.g. \cite{sciacchitano2015}.

One of the first second-order algorithms developed for solving composite problems was introduced in \cite{fletcher82} for a general composition of a smooth and a nonsmooth functions. In their approach, the cost function is approximated by smooth functions and then the surrogate smoothed model is solved using a trust-region algorithm. The surrogate model is itself a nonsmooth composite problem which is solved by expressing the nonsmooth penalization as a polyhedral function, leading to a constrained quadratic optimization problem. However, this procedure, in the case of the $\ell_1$-norm, requires a dense matrix of size $m\times 2^m$ to express $\norm{Cx}_1$ as a polyhedral function using the columns of $H$, which might be prohibitive for large values of $m$.

More recently, a primal-dual second-order method was proposed in \cite{dhingra17}. There, a new variable $y$, which represents the composite term, is introduced in order to cope with the penalization term as a constraint. This constraint is penalized by introducing an additional dual variable at the cost of increasing the size of the problem. The variant in this approach, formulated in \cite{dhingra17}, uses the proximal operator in order to represent the Lagrange function by means of its Moreau's envelope function. Then, a generalized second-order Hessian of the Lagrangian is introduced by computing the Clarke subgradient of the associated proximal operator. This generalized Hessian provides second-order updates for the primal and the dual variables. However, its second--order system still requires the computation of the proximal operator of the nonsmooth penalizer, which does not have a closed form in the case of the term $\norm{Cx}_1$.

Building up on the orthant--wise second-order method developed in \cite{dlrlm07}, we devise in this paper a new algorithm which utilizes second--order information from the regular part $f$ and also from the nondifferentiable composite term $\norm{C \cdot}_1$. As expected, the transformation of the variable by the pattern matrix $C$ entails new numerical and theoretical challenges, since the sparsity term is no longer separable. The main novelty to deal with this consists of the extension of generalized descent direction developed in \cite{dlrlm07}, by using second-order information associated with the nonsmooth term as well as the convergence analysis related to the proposed algorithm. Therefore, this paper contributes with a new efficient second--order algorithm to solve composite sparse optimization problems with well-founded theoretical properties. Indeed, by applying the techniques from \cite{attouch2009}, using the \L ojasiewikcz condition, we derive the corresponding convergence analysis of the proposed method.

We organize this paper by setting the problem in Section 2. In Section 3 we describe the different elements of the algorithm. Section 4 is devoted to the convergence analysis and the derivation of the corresponding rate. Finally, we present the numerical tests that shows how second--order information is relevant for the numerical performance.

\section{Problem formulation}
Let $f: \reals^m \arrow \reals$ be a differentiable function and let $\beta >0$. We are interested in the numerical solution of the unconstrained optimization problem
\begin{equation} \label{eq:P}
\min_{x \in \reals^m} \,\varphi(x):=f(x) + \beta \| C x \|_1, \tag{\bf{P}}
\end{equation}
where $\| \cdot \|_1$ corresponds to the standard $1$--norm in $\reals^m$ and $C$ is a real $n\times m$ matrix with rows $c_i$, for $i=1,\ldots,n$. We shall notice that by modifying the matrix $C$, problem \eqref{eq:P} also covers the so-called \emph{fused problem}
\begin{equation} \label{eq:P_F}
\min_{x \in \reals^m} \,\varphi(x):=f(x) + \alpha \norm{x}_1 + \beta \| C x \|_1. 
\end{equation}

In order to obtain existence of solutions for problem \eqref{eq:P} the following conditions are assumed. The existence of solutions then follows from Weierstrass' theorem.
\begin{assumption}{\label{h:f}}
\hspace{1mm}
\begin{enumerate}[(i)]
\item $f$ is bounded from below;
\item $f: \reals^m \arrow \reals$ is continuously differentiable, with  locally Lipschitz continuous gradient $\nabla f$;
\item $\varphi = f + \beta \norm{C\cdot}_1$ is coercive, i.e. $\lim_{\norm{x}\arrow \infty} \varphi(x) = + \infty$. 
\end{enumerate}
\end{assumption}

\subsection{First order optimality conditions} Let us denote by $\bar x$ the solution of  \eqref{eq:P} and by $\partial \phi (x)$ the subdifferential of the function $\phi$ at $x$. Moreover, let us denote by $g$ the convex nondifferentiable part of $\varphi$, that is $g(x) =  \beta \| C x \|_1$. By using the standard theory, the Fermat's condition gives the first-order necessary optimality conditions for \eqref{eq:P}:
\begin{equation}
0 \in \nabla f(\bar x) + \partial g(\bar x).	\label{eq:foc}
\end{equation}
By using subdifferential calculus rules, we may argue that if $\bar x$ is a solution for  \eqref{eq:P}, then there exists $\xi(\bar x) \in \reals^n$ such that:
	\begin{align}\label{eq:fonc}
		0 = \nabla f(\bar x)  + \beta C^\top \xi (\bar x),
	\end{align}
where the corresponding entries of $\xi( x) \in \partial \norm{\cdot}_1 (C  x)$ are given by
\begin{equation}
	\xi(x)_i
		 = \begin{cases}
			\{\sign( \langle c_i,   x\rangle)\}, & \text{ if }  \langle c_i,   x\rangle \not =0,\\
			[-1,1], & \text{ if }  \langle c_i,   x\rangle  =0.
		\end{cases} \label{eq:sd2}
\end{equation}
For a given $x$, let us define the index sets
\begin{align*}
 {\mathcal P}=\{i:  \langle c_i , x \rangle>0 \}, \quad  {\mathcal N}=\{i: \langle c_i , x \rangle<0 \}, \quad \text{and } {\mathcal A}=\{i: \langle c_i , x\rangle=0 \}.
\end{align*}
Then, condition \eqref{eq:fonc} is equivalent to the existence of $\bar{\xi}_i: = \xi (\bar x)_i$, for $i\in \bar {\mathcal A}$, such that
\begin{equation}
- \sum_{i\in \bar {\mathcal A}}\bar{\xi}_i c_i^\top = \frac{1}{\beta}\nabla f(\bar x ) +  \sum_{i\in \bar {\mathcal P}} c_i^\top -  \sum_{i\in\bar {\mathcal N}} c_i^\top , \label{eq:optsys}
\end{equation}
where $\bar {\mathcal P}$, $\bar {\mathcal N}$ and $\bar {\mathcal A}$ are the corresponding index sets associated to $\bar x$.

Notice that the linear system \eqref{eq:optsys} is of size $m \times p$, with $p\leq n$ being the cardinality of $\bar{\mathcal A}$. Let us denote by $C_{\mathcal{A}} \in \reals^{m\times p}$ the matrix whose columns are formed by the transposed rows indexed in $\mathcal A$ and by $\tilde C_{\bar{\mathcal{A}}}$ the corresponding augmented matrix, i.e. the matrix with the extra column given by the right--hand side of \eqref{eq:optsys}. In the following, we will assume the Rouch\'e--Capelli theorem holds. That is, the system \eqref{eq:optsys} has at least one solution provided that $\text{rank} \{\tilde C_{\bar{\mathcal{A}}} \} = \text{rank} \{ C_{\bar{\mathcal{A}}} \}$.

\section{Second-order algorithm}
We start with the construction of a \emph{descent direction}, for which we consider a vector of the form $\nabla f(x) + \beta C^\top \xi(x)$ according to \eqref{eq:sd2}.

\subsection{Computation of a descent direction}
In standard 1--norm penalized problems \cite{dlrlm07}, the natural choice for the subgradient element is the one with the minimum 2-norm or, equivalently in the convex case, the steepest descent direction \cite{sra2012optimization}. Because of the particular structure of the 1--norm, the minimum norm subgradient is also known as \emph{orthant direction}. In fact, it characterizes the orthant in which a descent direction has to be found.

However, in the case of composite optimization, the term $\norm{Cx}_1$ is no longer separable. Therefore, there is no orthant--wise interpretation for the minimum norm subgradient, which is defined in general as:
%
\begin{equation}\label{eq:mns.1}
 \xi^*(x)  \in  \displaystyle \text{argmin} \{ \|\nabla f(x)+  \beta C^\top \xi \|_2: \xi \in \partial \norm{\cdot}_1 (C  x)
  \}
\end{equation}

One of the drawbacks of using the minimum norm subgradient is that its computation requires the solution of an auxiliary quadratic optimization problem with box constraints. However, although an additional optimization subproblem is needed, it is not as expensive as it may appear at first sight. Indeed, since we already know that $\xi_i=\sign ( \langle c_i , x \rangle)$, if $\langle c_i , x \rangle \not =0$, we can exclude these components in the optimization problem \eqref{eq:mns.1}.

Let  $p:=|\mathcal A|$ and let us denote
$$\tilde \nabla \varphi (x):= \nabla f(x) +\beta \sum_{i\in \mathcal P} c_i^\top - \beta \sum_{i\in \mathcal  N} c_i^\top.$$
Further, let $C_{\mathcal A}$ denote the matrix obtained by removing all rows $c_i$, with $i\in \mathcal N \cup \mathcal P$, from $C$. Hence, we may reformulate problem  \eqref{eq:mns.1} as the following box--constrained quadratic optimization problem:
\begin{align}
	\min_{ \tilde \xi \in [-1,1]^p}\, \frac12 \big\| \tilde \nabla \varphi (x)
	 + \beta C_{\mathcal A}^\top \tilde \xi \big\|_2^2 \tag{MinSub}\label{eq:od}
\end{align}
Notice that this problem is of the same size as the active set cardinality at $x$. In many cases $C_{\mathcal A}C_{\mathcal A}^\top$ is nonsingular, thus problem \eqref{eq:od} has a unique solution. Moreover, the solution of \eqref{eq:od} is given by
\begin{equation}\label{eq:xi_projected}
{ \tilde \xi = \ds\mathbb{P}_{[-1,1]^p}{\{\tilde \xi-\beta C_{\mathcal A}\tilde \nabla \varphi (x)-\beta^2 C_{\mathcal A}C_{\mathcal A}^\top \tilde \xi\}, }}
\end{equation}
where $\mathbb{P}_{I}$ denotes the projection on a set $I$. Formula \eqref{eq:xi_projected} cannot be  computed as a closed--form solution. Indeed, its dual fits in a classical LASSO problem formulation. We will discuss the numerical solution for this problem in Section ...

\subsection{Second order information}\label{s:2ndorder}
Weak second-order information associated to the $1$--norm was algorithmically introduced in \cite{dlrlm07} in order to compute generalized hessian based descent directions that incorporate components coming from both the smooth and nonsmooth terms. There, the regularization of the $\ell_1$--norm by Huber smoothing allowed to obtain the targeted second-order information using the second derivative of its regularization. This procedure is analogous to consider generalized Hessians in the Bouligand subdifferential of the proximal operator $\partial_B\prox_{\frac{1}{\gamma} \, \| \cdot \|_1}$, see .

Here, we generalize this procedure to the case of composite sparse optimization. In the present case, however, the weak second order derivative of the nondifferentiable term is no longer a diagonal matrix. Indeed, recalling that the Huber regularization of the $1$--norm, for $\gamma>0$, is defined by
\begin{equation}\label{e:huber}
h_\gamma(x_i)=
\begin{cases}
\gamma \frac{x_i^2}{2} & \hbox{if } |x_i|\leq \frac{1}{\gamma},\\
|x_i|- \frac{1}{2\gamma}&\hbox{if } |x_i|> \frac{1}{\gamma},
\end{cases}
\end{equation}
we now regularize $\norm{C \cdot}_{1}$ as follows:
\[
h_\gamma(C x)=
\begin{cases}
\frac{\gamma}{2} \langle c_i,x \rangle^2  & \hbox{if } |\langle c_i,x \rangle|\leq \frac{1}{\gamma},\\
|\langle c_i,x \rangle|- \frac{1}{2\gamma}&\hbox{if } |\langle c_i,x\rangle|> \frac{1}{\gamma}.
\end{cases}
\]
Then, $\nabla h_\gamma(C u)$ is given by
\begin{equation}\label{eq:Gradh}
\nabla h_\gamma(C x)
=C^\top \l[ \frac{\langle c_i,x \rangle}{\max{\{1/\gamma, \langle c_i,x \rangle \}}} \r]_{i=1}^m ,
\end{equation}
and the ``weak Hessian" of $\norm{C \cdot}_{\ell_1}$ is given by the matrix
\begin{equation}\label{eq:Gamma} \Gamma =  \gamma C^\top D C, \quad  \text{with } \, D = \text{diag} \l( \l[ \begin{cases} 1 & \text{ if } |\langle c_i,x \rangle| \leq \frac{1}{\gamma} \\ 0 & \text{otherwise}\end{cases} \r]_{i=1}^{i=n} \r)
\end{equation}
By recalling the fact that $\prox_{\frac{1}{\gamma}\norm{\cdot}_1}$ is equal to the \textit{soft-thresholding} operator (e.g. see \cite{claval2020}), one could relize that $D \in I - \partial_B(\prox_{\frac{1}{\gamma}{\|\cdot\|_1}([\langle c_i,x\rangle]_{i=1}^n)})$, where $\partial_B$ denotes de Bouligand's subdifferential. 

We will write $\Gamma^k$ to specify that \eqref{eq:Gamma} is computed for $x=x^k$.
Now, the computation of the descent direction is carried on with help of the matrix in \eqref{eq:Gamma}, requiring the solution of the following linear system:
\begin{equation}\label{eq:direction}
\left[B^k + \beta \Gamma^k\right] d^k = - [{\nabla}f(x^k) + \beta C^\top\xi(x^k)],
\end{equation}
where $B^k$ stands either for the Hessian of $f$ at $x^k$ or an  approximation of it.

\begin{assumption}\label{h:B_spd}
	The matrix $B^k$ is symmetric positive definite and satisfies
	\begin{equation}
 \kappa \|d \|_2^2 \leq d^\top B^k d \leq K \|d \|_2^2,
 \label{eq:Bk_bounds}
\end{equation}
for all $d \in \reals^m$ and for some constants $K, \kappa >0$.
\end{assumption}

\subsection{Projection step}
In our algorithm, at each iteration, the approximated solution $x$ may be close to fulfill sparsity in the range of $C$, i.e., $\langle c_i, x \rangle \approx 0$ for some of the indexes $i$. However, small perturbations on $x$ may cause undesired sign changing in $\langle c_i, x \rangle$. When, under small perturbations on $x$, a change in the sign of the quantity $\langle c_i, x \rangle$ is detected, we might prefer to keep the updated approximated solution satisfying the sparsity condition. To achieve this, we consider a projection of $x$ to the closest point $\tilde x$ satisfying $\langle c_i, \tilde x \rangle =0$.

Thus, for a given approximated solution $x$ and a descent direction $y$, we identify those $\langle c_i, x \rangle$ which change sign with respect to the subgradient $\xi (x)$ (recall that the subgradient $\xi (x)$ has the same sign of $\langle c_i, x \rangle$, when it is not 0).
For the sign identification process we introduce the set
\begin{equation}\label{eq:esign_set}
\mathcal{S}(y) = \{ i =1,\ldots,n: \sign  ( \langle c_i, y \rangle) \not = \sign({\xi _i (x)})\},
\end{equation}
and define  $C_{s}:= C(\mathcal{S}(y),:)$. Then, we consider the projection over the subspace $ \mathcal{A}_S$, defined by
\begin{equation}
\mathcal{A}_S = \{ y \in \reals^m:  \langle c_i , y \rangle=0, \text{for } i \in \mathcal{S}(y) \}
 \end{equation}
Thus, the projection $\mathcal P$ on the set $\mathcal A_S$ is obtained as the solution of the following problem:
\begin{equation}\label{eq:proj} \tag{Prj}
\min_{\tilde x \in \mathcal{A}_S} \frac12\norm{\tilde x - x}^2_2 \Leftrightarrow  \min_{C_{s}\, \tilde x  = 0} \frac12\norm{\tilde x - x}^2_2
\end{equation}
%

It is known that \eqref{eq:proj} is  a saddle point problem. A particular but important case is when $C_s$ has full rank. Then, \eqref{eq:proj} is equivalent to the linear system (see \cite{BenGolLie05})
{
\begin{equation}\label{eq:fullrank proj}
	\l[\begin{array}{cc}
		I & C_{s}^\top \\
		C_ {s} & O
	\end{array} \r] \l[\begin{array}{c} \tilde x\\y \end{array}\r] =
	\l[\begin{array}{cc}
		x \\
		0
	\end{array} \r].
	\end{equation}}
Furthermore, by introducing the projections $\Pi:= C_{s}^\top ( C_{s}C_{s}^\top)^{-1} C_{s} $ and $\mathcal P=I-\Pi$, we can solve \eqref{eq:fullrank proj} explicitly and the solution of \eqref{eq:proj} reads:
\begin{subequations}\label{eq:explicit fullrank proj}
\begin{align}
	\tilde x & = \mathcal{P}\,x = x - \Pi x \quad {\in \text{span}\{c_i:i\in \mathcal{S}\}^{\perp}, } \label{eq:Pr}\\
	y & = ( C_{s}C_{s}^\top)^{-1} C_{s} \,x. %
\end{align}
\end{subequations}

Note that, $\Pi x$ is characterized as the solution of
\begin{equation}
	\min_{z \in \text{range} C_{s}^\top } {\norm{x -z}^2}.
\end{equation}
 Moreover, feasibility of $\tilde x$ implies that $C_{\mathcal S} \tilde x =0 $. From these relations, we realize that $x = \tilde x + \Pi x$, that is, $x \in \text{span}\{c_i:i\in \mathcal{S}\}^{\perp} \oplus \text{span}\{c_i:i\in \mathcal{S}\}$. In other words, the projection step removes the part belonging to $\text{range}(C_{s})$ from the current approximation.

 In the case that $C_S$ is not full rank, it cannot be guaranteed the existence of $( C_{s}C_{s}^\top)^{-1}$. Then, the common practice is to consider instead a regularization $C_{s}C_{s}^\top + \epsilon I$ for small $\epsilon>0$.

\subsection{Linesearch step}
Analogously to \cite{andrgao07,byrd011}, we consider a projected line-search rule using $\mathcal P$ given by \eqref{eq:Pr}, for choosing the step $s_k$ fulfilling the decrease condition:
\begin{equation} \label{eq:line search}
\varphi[\mathcal{P}(x^k + s_k d^k)]\leq \varphi(x^ k) + \widetilde{\nabla} \varphi(x^k)^ T[\mathcal{P}(x^k+s_k d^k)-x^k].
\end{equation}
The calculation of the step $s_k$ fulfilling the last condition is performed using a backtracking scheme.
\begin{algorithm}[H]\label{alg:A1}
\caption{Second--Order Method for Sparse Composite Optimization}
\begin{algorithmic}[1]
\STATE Initialize  $x^0$.
\WHILE{\hbox{stoping criteria is false}}
\STATE Compute $\xi^k$ given by solving \eqref{eq:od}
\STATE Compute $d^k$  by solving system \eqref{eq:direction}
\STATE Compute $s_k$ using a line--search procedure
\STATE Update $x^{k+1} \gets \mathcal{P}(x^k + s_kd^k)$
\STATE $k \gets k+1$.
\ENDWHILE
\end{algorithmic}\label{alg:A1}
\end{algorithm}

\subsection{Active--set identification strategy} \label{sec: active identification}

Second-order methods are known to be expensive when it comes to the computation of a descent direction. Without any additional strategy regarding the numerical solution of system \eqref{eq:direction}, the method would hardly become practical for large problems. Therefore, it is important to look at the structure of the pattern matrix $C$ and take it into account in order to improve the computation process.

In an effort to reduce the numerical cost, we extend the definition of \emph{active sets} used in \cite{dlrlm07} in order to define an effective identification process of the components of the optimization variable which are known to fulfill optimality conditions and therefore, can be excluded when seeking for a descent direction. In this way, the optimization process takes place in a lower dimensional subspace, resulting in a significant reduction of the computation cost.

A common situation occurs when the matrix $C$ possesses a known structure e.g., when $C$ is the successive difference matrix or ``discrete gradient"; in this case, $C$ is a banded matrix. We notice that in the multiplication $Cx^k$, not all the entries of $x^k$ are taking part in the computation of a particular component of the product $Cx^k$.

Recalling the optimality condition \eqref{eq:optsys}, for each $i \in \mathcal{A}^k$ we consider the  index set denoted by $\mathcal{I}^k_i$, consisting of indexes $j\in \{1,\ldots,m \}$ such that  $c_{ij} \not = 0$ and
\begin{equation}\label{eq:secured}
	|[\nabla f(x^k) + \beta C^\top \xi^k]_{j}|\approx 0.
\end{equation}
Then, we define the set of active entries of $x^k$ by
\begin{equation} \label{eq:I_0}
\mathcal{I}^k_{0}:= \ds\cup_{i\in \mathcal{A}^k} \mathcal{I}^k_i,
\end{equation}
which corresponds to the set of indexes that are close to satisfy optimality conditions which are active. Thus, we would not move from the current approximation $x^k$ in the entries indexed by $\mathcal{I}^k_0$. By contrast, we define the set of indexes $\mathcal{I}^k_F:=\{1,\ldots,m\}\setminus \mathcal{I}^k_0$, in which the variable is free to move. Thus, we consider the reduced system:
\begin{equation}\label{eq:reduced}
\left[\tilde B^k + \beta \tilde \Gamma^k\right] \tilde d^k = - [{\nabla}f(x^k) + \beta C^\top\xi(x^k)]_{j\in \mathcal{I}_F},
\end{equation}
where
\begin{equation*}
	\tilde B^k := [B^k_{ij}]_{i \in \mathcal{I}_F,j\in \mathcal{I}_F}, \quad \text{and } \quad  \tilde \Gamma^k:=[\Gamma^k_{ij}]_{i \in \mathcal{I}_F,j\in \mathcal{I}_F}.
\end{equation*}
Then, step 4 of Algorithm \ref{alg:A1} can be modified  using \eqref{eq:reduced} and by choosing the descent direction $d$ computed according to the formula
\begin{equation} \label{eq:descendr}
	d_j =
	\begin{cases}
		\tilde d_j & \text {if } j\in  \mathcal{I}_F, \\
		0 &  \text {if } j\in  \mathcal{I}_0.
	\end{cases}
\end{equation}

\section{Convergence Analysis}
Let $x^k$ be the approximated solution computed by Algorithm \eqref{alg:A1} in the $k$-th iteration. Moreover,  let $C_{k}:=C_{\mathcal{S}^k}$, for $k=1,2,\ldots$, and $\xi^k:=\xi(x^k)$. Hence, at every step  $\Pi = C_k(C_kC_k^\top)^{-1}C_k$. In addition, for a vector $y\in \reals^m$, according to \eqref{eq:esign_set}, we consider the index set
\begin{equation}\label{eq:Sk}
\mathcal{S}_k = \{ i=1,\ldots,n: \sign \langle c_i , x^{k} + s d^k\rangle\not =\sign (\xi^k_i)\}.
\end{equation}
\begin{remark}\label{r:s_mall}
	It follows from the definition of $\mathcal{S}_k$ that for $s$ sufficiently small $x^k$ belongs to the null space of $C_{k}$ and the index set $\mathcal{S}_k$ may be equivalently defined as $$\mathcal{S}_k = \{  i=1,\ldots,n:   \sign (\xi_i) \, \sign\langle c_i, d^k \rangle \leq 0\}.$$
	Indeed, this can be seen from the fact that if $i\in \mathcal{S}_k$ then we have that if $\langle c_i, x^k \rangle \not = 0$ then $\xi^k_i = \sign{\langle c_i,x^k \rangle }$ and, for sufficiently small $s$, we have  $\sign \langle c_i,x^k +sd^k\rangle = \sign\langle c_i,x^k \rangle  \not = \xi^k_i$, which is a contradiction. Therefore, the only possibility is that $\langle c_i,x^k \rangle  =0$. Thus, $\sign\langle c_i, d^k \rangle \sign (\xi^k_i)\leq 0$.
\end{remark}
\begin{theorem}{\label{t:monotonicity}} Let Assumptions \ref{h:f} and \ref{h:B_spd} hold, and let $x^k$ be the approximated solution for \eqref{eq:P} at the $k$th iteration of Algorithm \ref{alg:A1} and let $d^k$ be the corresponding direction computed using \eqref{eq:direction}. Let us assume that $C_k$ defined in projection step \eqref{eq:proj} is full rank. Moreover, let us assume that at every step $\langle c_i, d^k\rangle \not =0$ for some i, and that the parameter $\gamma =\gamma_{k+1}$ is chosen in each iteration such that
	\begin{equation} \gamma_{k+1} > \frac{1}{2\beta} \l( \frac{\norm{|\nu^k|+\beta(|\xi^k|+n|\eta^k|)}^2}{\min{\langle c_i, d^k \rangle }^2} +1 \r), \label{eq:gamma_k} \end{equation}
	where the minimum is taken from those $\langle c_i, d^k \rangle \not =0$, where $\nu^k$ and $\eta^k$ being the vectors of coefficients of $\Pi \nabla f(x^k)$ and $\Pi c_{i^*}$ on $\text{span}\{c_i: i \in \mathcal{S}^k\}$, respectively.  Here $i^*$ is such that $|\langle c_{i^*}, \Pi d^k) \rangle|=\max_{i\in \mathcal{S}_k}|\langle c_{i}, \Pi d^k) \rangle|$. Then, 	$d^k$ is a descent direction, i.e.:
	\begin{equation}
	\varphi(x^{k+1})<\varphi(x^k).
	\end{equation}
\end{theorem}
\begin{proof}
Taking into account that $x^{k+1} ={P}(x^k+sd^k)= x^k+sd^k - \Pi (x^k+sd^k) $ and $C_k$ is full rank then, by \eqref{eq:fullrank proj}, it follows that $C_k x^{k+1} = 0$. That is, $\langle c_i, x^{k+1}\rangle = 0$ for all $i \in \mathcal S_{k}$. Moreover, if $i \in \mathcal{S}_k$ we have either $\langle c_i,x^k\rangle = 0$ or $\langle c_i,x^k\rangle \not = 0$. In the first case, it is clear that $0 = |\langle c_i,x^k\rangle| \leq s|\langle c_i,d^k\rangle|$. On the other hand, if $\langle c_i,x^k\rangle \not = 0$, we have that $\sign (\langle c_i,x^k + sd^k \rangle) \not = \sign(\xi^k) = \sign (\langle c_i,x^k\rangle)$. Then, we conclude that $|\langle c_i,x^k\rangle| < s |\langle c_i,d^k\rangle|$. Hence,
\begin{equation*}
\norm{C_k x^k} = \norm{ [\langle c_i,x^k \rangle]_{i \in \mathcal{S}_k} }	 \leq s \norm{C_k} \norm{d^k},
\end{equation*}
which implies that $\norm{\Pi x^k} \leq s \norm{C_k}^2 \norm{(C_kC_k^\top)^{-1}} \norm{d^k} \leq sc\norm{d^k}$, for some constant $c$ depending on the matrix $C$ and independent of $k$. Therefore, we obtain the estimate
\begin{align}
	\norm{x^{k+1} -x^k} & = \norm{s{\mathcal P}d^k - \Pi x^k } \nonumber\\
						& \leq s\norm{\mathcal{P}d^k} + \norm{\Pi x^k} \nonumber\\
						& \leq s(1+c) \norm{d^k}. \label{eq:descent_0}
\end{align}
Now, using  \eqref{eq:descent_0} and the first order Taylor expansion of the regular part of $\varphi$, we get
\begin{align}
	\varphi(x^{k+1}) - \varphi(x^{k}) & = f(x^{k+1}) - f(x^{k}) + \beta \norm{Cx^{k+1}}_1 -\beta \norm{Cx^{k}}_1 \nonumber \\
	& = \nabla f(x^k)^\top \l(\mathcal{P}(x^k+sd^k) -x^k \r)+ o(s \norm{d^k})\nonumber \\
	&\qquad+ \beta \sum_{i} \l(|\langle c_i, x^{k+1} \rangle|- |\langle c_i, x^{k} \rangle| \r).\label{eq:descent_1}	
\end{align}
From the second--order system \eqref{eq:direction} and the positive semidefiniteness of $\Pi$, we see that $x^{k+1}-x^k =\mathcal P (x^k +sd^k) -x^k = sd^k-\Pi( x^k +sd^k) $, therefore
\begin{align}
\nabla f(x^k)^\top(\mathcal{P}(x^k+sd^k) -x^k) 
						 =  & s\nabla f(x^k)^\top d^k - s\nabla f(x^k)^\top\Pi d^k- \nabla f(x^k)^\top\Pi x^k   \nonumber \\
						 =  -&s{d^k}^\top \l[ B^k + \beta \Gamma^k \r]d^k - s\beta {\xi^k}^\top C d^k -    \,\nabla f(x^k)^\top \Pi(x^k+sd^k). \label{eq:descent_2}
\end{align}
%
Note that $\Pi=\Pi^2$; moreover, it is also a symmetric positive semi--definite matrix. In addition, we have that $\Gamma^k=\gamma C^\top D^k C$ is symmetric and positive semidefinite by its construction. Further, by Assumption \ref{h:B_spd} we have that  exists a positive constant $\hat c$, independent of $k$, such that ${d^k}^\top B^k d^k \geq \hat c \norm{d^k}^2$. Therefore, these matrix properties imply
\begin{align}
\nabla f(x^k)^\top({P}(x^k+sd^k) -x^k)  \leq &
-{d^k}^\top B^k d^k - {s}\gamma\beta (Cd^k)^\top D^k (Cd^k)  - s\beta {\xi^k}^\top C d^k \nonumber \\
&- \,\nabla f(x^k)^\top \Pi(x^k+sd^k) \nonumber \\
\leq &- s{\hat c} \norm{d^k}^2 - \gamma s \beta \sum_{i: |\langle c_i, x^k \rangle|\leq 1/\gamma} \langle c_i,d^k \rangle^2 -  s \beta \sum_{\substack{i\in \mathcal{S}_k 
}
}
\xi_i^k\langle c_i, d^k \rangle
 \nonumber \\
 & -s\beta\sum_{i\not \in \mathcal{S}_k} \xi^k_i \langle c_i, d^k \rangle  - \,\nabla f(x^k)^\top \Pi(x^k+sd^k).\label{eq:descent_3}
\end{align}

 Let us focus on the sum on the right--hand side of  \eqref{eq:descent_1}. Since for all $i \in \mathcal{S}_k = \{ i \in\{1,\ldots,n \}: \sign \langle c_i , x^{k} + s d^k\rangle\not =\sign (\xi^k_i)\}$, we have $\langle c_i, x^{k+1}\rangle = 0$; then:
 \begin{align}
 \nonumber\sum_{i} (|\langle c_i, x^{k+1} \rangle|- |\langle c_i, x^{k} \rangle| )	&= \sum_{i\not \in \mathcal{S}_k} (|\langle c_i, x^{k+1} \rangle|- |\langle c_i, x^{k} \rangle| ) - \sum_{i \in \mathcal{S}_k} |\langle c_i, x^{k} \rangle| \\
 	&= \sum_{i\not \in \mathcal{S}_k} (|\langle c_i, \mathcal{P}(x^k +sd^k) \rangle|- |\langle c_i, x^{k} \rangle| ) - \sum_{i \in \mathcal{S}_k} |\langle c_i, x^{k} \rangle| \nonumber \\
 	&\leq \sum_{i\not \in \mathcal{S}_k} |\langle c_i, x^k +sd^k \rangle| + |\langle c_i, \Pi(x^k +sd^k) \rangle|  - |\langle c_i, x^{k} \rangle|  \nonumber \\
 	&\leq \sum_{i\not \in \mathcal{S}_k} \xi^k_i\langle c_i, x^k +sd^k \rangle + |\langle c_i, \Pi(x^k +sd^k) \rangle|  - |\langle c_i, x^{k} \rangle| \nonumber \\
 	&\leq \sum_{i\not \in \mathcal{S}_k} \xi^k_i\langle c_i, sd^k \rangle + |\langle c_i, \Pi(x^k +sd^k) \rangle|  \nonumber
\end{align}
 {
 Using  Remark \ref{r:s_mall},  it follows that $C_kx^k =0$ if $s$ is small enough, hence
 \begin{align}
\sum_{i} (|\langle c_i, x^{k+1} \rangle|- |\langle c_i, x^{k} \rangle| )
 	&\leq \sum_{i\not \in \mathcal{S}_k} \xi^k_i\langle c_i, sd^k \rangle + s|\langle c_i, \Pi d^k) \rangle| \label{eq:descent_4}.
 \end{align}
Inserting \eqref{eq:descent_3} and \eqref{eq:descent_4} in \eqref{eq:descent_1}	obtain the relation:
\begin{align}
	\varphi(x^{k+1}) - \varphi(x^{k})  \leq &- s\hat c \norm{d^k}^2 - \gamma s \beta \sum_{i: |\langle c_i, x^k \rangle|\leq 1/\gamma} \langle c_i,d^k \rangle^2 -  s \beta \sum_{\substack{i\in \mathcal{S}_k 
}
}
\xi_i^k\langle c_i, d^k \rangle\nonumber \\
	& + s\beta\sum_{i\not \in \mathcal{S}_k} |\langle c_i, \Pi d^k) \rangle|-    \,\nabla f(x^k)^\top \Pi(x^k+sd^k) + o(s \norm{d^k})\nonumber\\
\leq &- s\hat c\norm{d^k}^2 - \gamma s \beta \sum_{i: |\langle c_i, x^k \rangle|\leq 1/\gamma} \langle c_i,d^k \rangle^2 + o(s \norm{d^k}) \nonumber \\
&\qquad - s \beta \sum_{\substack{i\in \mathcal{S}_k 
}
}
\xi_i^k\langle c_i, d^k \rangle
	 + s\beta |\mathcal{S}^C_k| |\langle c_{i^*}, \Pi d^k) \rangle|-    \,\nabla f(x^k)^\top \Pi(x^k+sd^k) , \label{eq:descent_5}
	\end{align}
where $|\mathcal{S}^C_k|$ denotes the cardinality of the complement of the set $\mathcal{S}_k$ and $i^*$ is the index where the term $|\langle c_{i}, \Pi d^k) \rangle|$ attains it maximum in  $\mathcal{S}_k$.

By using again Remark \ref{r:s_mall}, and taking into account that $\Pi$ projects onto $\text{span}\{c_i: i \in \mathcal{S}_k\}$, we can be estimate the last three terms as follows:
\begin{align}
\Big|\nabla f(x^k)^\top \Pi(x^k+sd^k) &+ s \beta \sum_{\substack{i\in \mathcal{S}^k 
}
}
\xi_i^k\langle c_i, d^k \rangle -s\beta |\mathcal{S}^C_k| |\langle c_{i^*}, \Pi d^k) \rangle|  \Big| \nonumber \\
&=\Big|[\Pi\nabla f(x^k)]^\top (x^k+sd^k) + s \beta \sum_{\substack{i\in \mathcal{S}_k \\ \langle c_i, x^k \rangle=0
}
}
\xi_i^k\langle c_i, d^k \rangle  -s\beta |\mathcal{S}^C_k| |\langle \Pi c_{i^*}, d^k) \rangle|\Big| \nonumber\\
& = \big|\sum_{i\in\mathcal{S}_k}\nu^k_i \langle c_i, x^k+sd^k\rangle - s \beta \sum_{\substack{i\in \mathcal{S}_k \\ \langle c_i, x^k \rangle=0
}
}
|\xi_i^k|\,|\langle c_i, d^k \rangle| -s\beta |\mathcal{S}^C_k| |\langle \sum_{i\in\mathcal{S}_k} \eta^k_i c_i, d^k) \rangle| \big| \nonumber  \\
& \leq  s\sum_{ \substack{i\in\mathcal{S}_k \\ \langle c_i,x^k\rangle=0 }}(|\nu^k_i| + \beta(|\xi^k_i| +|\mathcal{S}^C_k||\eta_i^k|))\, |\langle c_i,d^k\rangle| \nonumber\\
& \leq  s \Big( \sum_{ \substack{i\in\mathcal{S}_k \\ \langle c_i,x^k\rangle=0 }}\frac12(|\nu^k_i|+\beta (|\xi^k_i|+n|\eta_i^k|))^2 +  \frac12|\langle c_i,d^k\rangle|^2 \Big). \label{eq:desced_3}
\end{align}
Notice that we have assumed that the set $\{i: \langle c_i,x^k\rangle \leq{1/\gamma} \} \not = \emptyset$, otherwise the right--hand side of \eqref{eq:desced_3} vanishes. Using  $\gamma=\gamma_k$ given in \eqref{eq:gamma_k} in the last relation and inserting in \eqref{eq:descent_5}, we arrive to
\begin{align}\label{eq:qgrowth}
\varphi(x^{k+1}) -\varphi(x^k)   \leq &
- s\hat c \norm{d^k}^2  + o(s\norm{d^k}),
\end{align}
which allows us to conclude that $d^k$ is a descent direction.
}
\end{proof}

There are nonconvex problems for which Assumption \ref{h:B_spd} can not be fullfilled, e.g. when $f$ is concave. In this case, the last proof can be modified to cope with this situation. We will need the following  assumption.
\begin{assumption}\label{h:Bk2}
	The matrix $B^k$ satisfies
	\begin{equation}
 |d^\top B^k d| \leq \hat C \|d \|_2^2, \qquad \text{ for all } d \in \reals^m,
\end{equation}
for some positive constant $\hat C$.
\end{assumption}

\begin{theorem}
	Let Assumptions \ref{h:f} and \ref{h:Bk2} hold. Consider $x^k$, $\nu^k$ and $\eta^k$ as in Theorem \ref{t:monotonicity}. Moreover, assume in addition that there exist a constant $\tilde C>0$ such that
	\begin{equation}\label{eq:coercivity2}
			0<\tilde{C}\norm{d^k}_2^2 \leq \sum_{i: |\langle c_i, x^k \rangle|\leq 1/\gamma}{\langle c_i, d^k \rangle }^2,
	\end{equation}
	for every $k$, and that the parameter $\gamma$ is chosen at each iteration as follows
	\begin{equation} \gamma_{k+1} > \frac{1}{2\beta} \l( \frac{2\hat{C}\norm{d^k}_2^2}{\sum_{i: |\langle c_i, x^k \rangle|\leq 1/\gamma}{\langle c_i, d^k \rangle }^2}+ \frac{\norm{|\nu^k|+\beta(|\xi^k|+n|\eta^k|)}_2^2}{\min{\langle c_i, d^k \rangle }^2} +1 \r), \label{eq:gamma_k}
	\end{equation}
	then, 	$d^k$ is a descent direction, i.e.:
	\begin{equation*}
	\varphi(x^{k+1})<\varphi(x^k).
	\end{equation*}
\end{theorem}

\begin{proof}
	Following the same arguments and notation  of the proof of Theorem \ref{t:monotonicity}, we have that

	\begin{align}
	\varphi(x^{k+1}) - \varphi(x^{k})  \leq &-s{d^k}^\top \l[ B^k + \beta \Gamma^k \r]d^k  -  s \beta \sum_{\substack{i\in \mathcal{S}_k 
}
}
\xi_i^k\langle c_i, d^k \rangle\nonumber \\
	& + s\beta\sum_{i\not \in \mathcal{S}_k} |\langle c_i, \Pi d^k) \rangle|-    \,\nabla f(x^k)^\top \Pi(x^k+sd^k) + o(s \norm{d^k})\nonumber\\
\leq & s\hat{C} \norm{d^k}^2 - \gamma s \beta \sum_{i: |\langle c_i, x^k \rangle|\leq 1/\gamma} \langle c_i,d^k \rangle^2 + o(s \norm{d^k}) \nonumber \\
&\qquad - s \beta \sum_{\substack{i\in \mathcal{S}_k 
}
}
\xi_i^k\langle c_i, d^k \rangle
	 + s\beta |\mathcal{S}^C_k| |\langle c_{i^*}, \Pi d^k) \rangle|-    \,\nabla f(x^k)^\top \Pi(x^k+sd^k) , \label{eq:descent_6}
	\end{align}
By the estimate \eqref{eq:desced_3} and \eqref{eq:coercivity2}  we get

\begin{align}
	\varphi(x^{k+1}) - \varphi(x^{k})
	\leq
& s\hat{C} \norm{d^k}^2 - \gamma s \beta \sum_{i: |\langle c_i, x^k \rangle|\leq 1/\gamma} \langle c_i,d^k \rangle^2 + o(s \norm{d^k}) \nonumber \\
&\qquad + s \Big( \sum_{ \substack{i\in\mathcal{S}_k \\ \langle c_i,x^k\rangle=0 }}\frac12(|\nu^k_i|+\beta (|\xi^k_i|+n|\eta_i^k|))^2 +  \frac12|\langle c_i,d^k\rangle|^2 \Big) \nonumber \\
&\leq -  \frac{s}{2} \sum_{i: |\langle c_i, x^k \rangle|\leq 1/\gamma} \langle c_i,d^k \rangle^2 + o(s \norm{d^k}).
	\end{align}
Finally, the right--han side of the last relation  is negative for sufficiently small $s$.
\end{proof}

\begin{definition}
	We will say that a function $f$ is a KL--function if $f$ satisfies the Kurdyka--\L ojasiewicz inequality, that is: for every $y \in \reals$ and for every bounded subset $E \subset \reals^m$, there exist three constants $\kappa>0$, $\zeta >0$ and $\theta \in [0,1[$ such that for all $z \in \partial f (x)$
and every $x \in E$ such that $|f(x) - y| \leq \zeta$, it follows that
\begin{equation}
	\kappa |f(x) - y|^\theta \leq \norm{z}_2, \label{eq:KL}
\end{equation}
with the convention $0^0 =0$.
\end{definition}

\begin{theorem}\label{t:convergence} Suppose that Assumptions \ref{h:f}--\ref{h:B_spd} are satisfied and that $\varphi$ is a KL--function (i.e. satisfies the Kurdyka--\L ojasiewicz condition). Then, the sequence $\{x^k\}_{k\in \mathbb N}$ generated by Algorithm 1 converges to a point $\bar x$ such that  $0 \in \nabla f(\bar x)+ \beta \, C^\top \partial  \|\cdot\|_1(C\bar x)$.
\end{theorem}
\begin{proof}
The proof of this convergence result is analogous to the proof of Theorem 2 in \cite{dlrlm07}. Indeed, notice that the sequence $\{x^k\}_{k\in \mathbb N}$ lies in the level set $\{x: \varphi(x)\leq\varphi(x^0) \}$, which in view of Assumption \eqref{h:f} is compact. Moreover, by Theorem \ref{t:1}, for $s^k$ sufficiently small, there exists $\mu>0$ such that the sequence $\{\varphi(x^k)\}_{k\in \mathbb N}$ enjoys the property:
\begin{equation}\label{eq:conv_1}
	\mu \norm{d^k}_2^2 \leq f(x^k) + \beta \norm{Cx^k}_1 - f(x^{k+1}) -  \beta \norm{Cx^{k+1}}_1,
\end{equation}
and $\varphi(x^k)$ converges to some value $\varphi_\infty$ as $k \arrow \infty$. By using the Kurdyka--\L ojasiewicz condition and Assumption \ref{h:B_spd}, there exist $\kappa>0$ and $\theta \in [0,1)$ such that
\begin{equation}\label{eq:conv_2}
 \kappa |\varphi(x^k) -\varphi_\infty|^\theta 	\leq \norm{ \nabla f(x^k) + \beta C^\top \xi^k}_2\leq \frac{\hat C}{\kappa}\norm{d^k}_2, \quad \forall \,\xi  \in \partial \l(\beta \|\cdot\|_1 \r)(C x^k) .
\end{equation}
holds. Therefore, majoring \eqref{eq:conv_1} using \eqref{eq:conv_2} it can be concluded the summability of the sequence $\{\norm{d^k}\}_{k\in \mathbb N}$. Which in turn, by \eqref{eq:descent_0}, implies that  $\{x^k\}_{k\in \mathbb N}$ is a Cauchy sequence and thus convergent. Let us denote its limit by $\bar x$.

Since $\nabla f(x^k) + \beta C^\top \xi^k   \in \nabla f(x^k) + \beta C^\top \partial \norm{\cdot}_1) (Cx^k)$ then we have
$$(x^k,  \nabla f(x^k) + \beta C^\top \xi^k) \in \text{Graph} (\nabla f+ \beta C^\top \partial \norm{\cdot}_1 (C \cdot) ) $$
Finally, using \eqref{eq:conv_2} and taking the limit $k\arrow \infty$ we obtain
\begin{equation*}
(x^k, \nabla f(x^k) + \beta C^\top \xi^k ) 	 \arrow (\bar x, 0) \quad\text{as} \quad k \arrow +\infty.
\end{equation*}
Hence $(\bar x, 0)$ belongs to  $\text{Graph} ( \nabla f + \partial (\beta \norm{\cdot}_1)  )$ due to its closedness which is equivalent to the relation $0 \in \nabla f(\bar x)+ \partial (\beta \norm{\cdot}_1) (\bar x)$.
\end{proof}

\begin{theorem}[Rate of convergence] Let Assumptions \ref{h:f}--\ref{h:B_spd} hold and assume also that $\varphi$ is a KL--function with \L ojasiewicz exponent $\theta \in (0,1)$. Let $\{ x^k\}_{k\in\mathbb{N}}$ be a sequence generated by Algorithm \ref{alg:A1}, converging to a local solution $\bar x$. Then, the following rates hold:
\begin{enumerate}[(i)]
	\item If $\theta \in (0,\frac12)$, then there exist $c>0$ and $\tau \in [0,1)$ such that $$\norm{x^k-\bar x} \leq c \tau^k$$.
	\item If $\theta \in (\frac12,1)$, then there exist $c>0$  such that $$\norm{x^k-\bar x} \leq c k^{-\frac{1-\theta}{2\theta-1}}.$$
\end{enumerate}

\end{theorem}

\begin{proof}
	We follow the ideas from \cite{attouch2009}. From \eqref{eq:descent_0} and the quadratic growth \eqref{eq:qgrowth}, for sufficiently small $s$, there is a positive constant $c$ such that
\begin{align}\label{eq:qgrowth1}
\norm{x^{k+1}-x^k}^2_2 \leq c\norm{d^k}^2   \leq &  \varphi(x^k) -\varphi(x^{k+1}),
\end{align}
Without loss of generality, we assume that $\varphi(\bar x) =0$ (we can always replace $\varphi(\cdot)$ by $\varphi(\cdot) - \varphi(\bar x)$ ) and by multiplying  relation \eqref{eq:qgrowth1} by $\varphi(x^k)^{-\theta}$
and using the fact that the real function $\reals_+\ni t \mapsto t^{1-\theta}$ is a concave differentiable function
\begin{align}
\norm{x^{k+1}-x^k}^2_2 \varphi(x^k)^{-\theta}   \leq &  (\varphi(x^k) -\varphi(x^{k+1}))\varphi(x^k)^{-\theta}\nonumber\\
\leq & \frac{1}{1-\theta}(\varphi(x^k)^{1-\theta} -\varphi(x^{k+1})^{1-\theta}). \nonumber
\end{align}
On the other hand, $\varphi$ is a KL--function  thus, from the last relation, we get
\begin{align}
\norm{x^{k+1}-x^k}^2_2
\leq & \frac{1}{1-\theta}(\varphi(x^k)^{1-\theta} -\varphi(x^{k+1})^{1-\theta})\varphi(x^k)^{\theta} \nonumber\\
\leq & \frac{1}{1-\theta}(\varphi(x^k)^{1-\theta} -\varphi(x^{k+1})^{1-\theta})\norm{\nabla f(x^k)+\beta C^\top \xi^k}_2.\label{eq:qgrowth3}
\end{align}
Fhurther, $\xi^k$ corresponds to the minimum norm subgradient solving  \eqref{eq:od}; therefore, by feasibility of $\xi^{k-1}$ we have that  $\norm{\nabla f(x^k)+\beta C^\top \xi^k}_2 \leq \norm{\nabla f(x^k)+\beta C^\top \xi^{k-1}}_2$ which can be inserted in \eqref{eq:qgrowth3} and combined with  \eqref{eq:direction} and Assumption \ref{h:f} to obtain that
\begin{align}
\norm{x^{k+1}-x^k}^2_2
\leq & \frac{1}{1-\theta}(\varphi(x^k)^{1-\theta} -\varphi(x^{k+1})^{1-\theta})\norm{\nabla f(x^k)+\beta C^\top \xi^{k-1}}_2\nonumber \\
\leq &\frac{1}{1-\theta}(\varphi(x^k)^{1-\theta} -\varphi(x^{k+1})^{1-\theta})\Big(\norm{\nabla f(x^k)-\nabla f(x^{k-1})}_2 \nonumber\\
&\quad+ \norm{\nabla f(x^{k-1})+\beta C^\top \xi^{k-1}}_2 \Big). \nonumber\\
\leq &\frac{1}{1-\theta}(\varphi(x^k)^{1-\theta} -\varphi(x^{k+1})^{1-\theta})(L_f\norm{x^k-x^{k-1}}_2 + \frac{\hat C}{\kappa}\norm{d^{k-1}}_2 ). \nonumber
\end{align}
As before, we invoke Remark \ref{r:s_mall} to infer that for sufficiently small $s$ it follows that $\Pi x^{k-1}=0$ then $s\norm{d^{k-1}}_2\leq s\norm{\mathcal{P}d^{k-1}}_2 = \norm{x^k-x^{k-1}+\Pi x^{k-1}  }_{2}\leq  \norm{x^k-x^{k-1}}_2$ which together with the above inequality imply that there exist a constant $c>0$ such that
\begin{align}
2\norm{x^{k+1}-x^k}_2
\leq &2\l(\frac{c}{1-\theta}(\varphi(x^k)^{1-\theta} -\varphi(x^{k+1})^{1-\theta})\r)^{\frac12}\norm{x^k-x^{k-1}}^{\frac12}_2. \nonumber \\
\leq &\frac{c}{1-\theta}(\varphi(x^k)^{1-\theta} -\varphi(x^{k+1})^{1-\theta}) +  \norm{x^k-x^{k-1}}_2.\nonumber\\
\leq &M_\theta(\varphi(x^k)^{1-\theta} -\varphi(x^{k+1})^{1-\theta}) + \norm{x^k-x^{k-1}}_2,\label{eq:rate1}
\end{align}
where $M_\theta$ is  a positive constant depending on $\theta$. Let us sum \eqref{eq:rate1} over $k$ from $k=n$ up to $N>k$:
\begin{align}
	\sum_{k=n}^N\norm{x^{k+1}-x^k}_2 + \norm{x^{N+1}-x^N}_2
\leq & M_\theta (\varphi(x^n)^{1-\theta} -\varphi(x^{N+1})^{1-\theta}) + \norm{x^n-x^{n-1}}_2,\nonumber
\end{align}
hence, recalling Theorem \ref{t:convergence} that $\{\norm{x^{k+1}-x^k}_2\}_{k\in \mathbb{N}}$ is summable in virtude of the sumability of the sequence $\{\norm{d^k}_2\}_{k\in \mathbb{N}}$ and taking $N \arrow \infty$, we get
\begin{align}
	\sum_{k=n}^\infty\norm{x^{k+1}-x^k}_2
\leq & M_\theta \varphi(x^n)^{1-\theta}  + \norm{x^n-x^{n-1}}_2. \nonumber
\end{align}
The last relation in terms of $\Delta^n:= \sum_{k=n}^\infty\norm{x^{k+1}-x^k}_2 $ can be rewritten as follows:
\begin{align}
	\Delta^n
\leq & M_\theta \varphi(x^n)^{1-\theta}  +  \Delta^{n-1}-\Delta^n. \nonumber\\
\leq & M_\theta \varphi(x^{n-1})^{1-\theta}  +  \Delta^{n-1}-\Delta^n.\label{eq:rate2}
\end{align}
Using again that $\varphi$ is a KL--function, we have from \eqref{eq:KL} and monotonicity that $\varphi(x^{n-1})^{1-\theta} \leq \frac{1}{\kappa}{\norm{\nabla f(x^{n-1}) + \beta C^\top \xi^{n-1}}_2}^{\frac{1-\theta}{\theta}} $. Thus, observing that $\Delta^{n-1}-\Delta^n = \norm{x^n-x^{n-1}}_2$, we obtain
\begin{align}
	\Delta^n
\leq & \frac{M_\theta}{\kappa}{\norm{\nabla f(x^{n-1}) + \beta C^\top \xi^{n-1}}_2}^{\frac{1-\theta}{\theta}}    +  \Delta^{n-1}-\Delta^n. \nonumber\\
\leq & M (\Delta^{n-1}-\Delta^n)^{\frac{1-\theta}{\theta}}+\Delta^{n-1}-\Delta^n, \label{eq:rate3}
\end{align}
where $M$ is a positive constant. Here, we rely on the analysis of a sequence satisfying relation \eqref{eq:rate3} done in \cite[pg. 13--15]{attouch2009} henceforth  (i) and (ii) hold.
\end{proof}

\section{Numerical experiments}
In this section we carry out some numerical experiments to show the performance of the proposed algorithm. Three application examples of the generalized 1--norm penalization are considered in order to illustrate the type of problems that can be handled with our algorithm.

The algorithm was implemented in Matlab. The \eqref{eq:od} problem of step 3 was solved by using \texttt{quadprog} package from the optimization toolbox whereas the linear system \eqref{eq:direction} of step 4 was solved using direct methods or iterative methods, depending on the matrix of system \eqref{eq:direction}, see experiments below.  In step 6 we implemented the line--search using a projected backtracking algorithm, by checking condition \eqref{eq:line search}. For the stopping criteria we use a given tolerance for the difference of consecutive values for the approximated solution and its corresponding costs. In the numerical experiments we compare our method with different algorithms designed specifically for the problem structure under consideration.

\subsection{Anisotropic total variation in function spaces}
We consider the following simplified version of an anisotropic viscoplastic fluid flow model:
\begin{equation}
\label{anisotrop}	\min_{ u  \in H_0^1(\om)} \frac12\int_\om |\nabla u|^2 dx - \int_\om z u \,dx + \beta\int_{\om}| \nabla (u)|_1 \,dx.
\end{equation}
After discretizing using finite differences, the infinite-dimensional problem is reformulated as an energy minimization problem of the form \eqref{eq:P}. Hence, the regular part $f(u) = \frac12 u^\top A u - b^\top u $ of our minimization problem is written using the matrix $A$ associated to the discrete laplacian and $b$ is the vector corresponding to the discretization of the forcing term $z$.
\begin{figure}[ht!]
\begin{subfigure}{0.45\textwidth}
\includegraphics{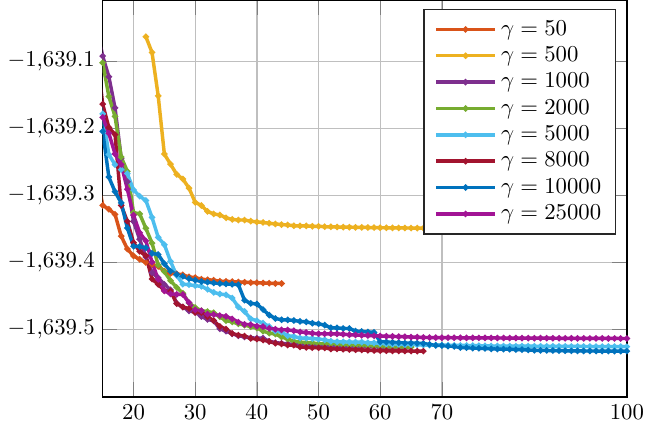}
%
\caption{Evolution of the cost function along the iterations, for different values of  $\gamma$.}
\end{subfigure}
\hfill
\begin{subfigure}{0.45\textwidth}
\includegraphics{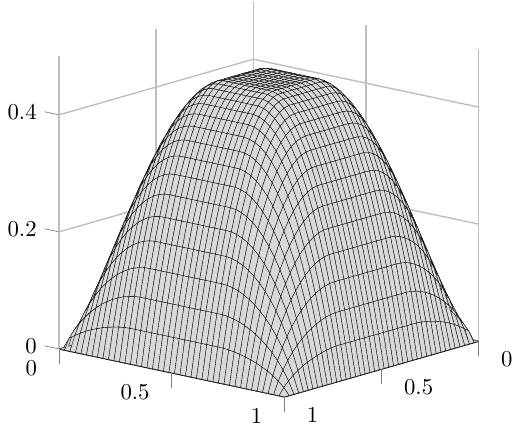}
\hfill
\caption{Solution for problem \eqref{anisotrop} for $\beta=0.5$ and parameter $\gamma = 25000$}
\end{subfigure}
\caption{Anisotropic viscoplastic flow}
\label{fig:ga}
 \end{figure}

We observe in Table \ref{t:1} the effect of using the \emph{generalized second--order information} introduced in Section \ref{s:2ndorder}. The cost values of the objective function were computed by varying the regularization parameter $\gamma$ for different values of $\beta$, after 50 iterations of the algorithm. The first row (in red) shows the cost values achieved by the algorithm when no generalized second--order information is utilized for the computation of the descent direction ($\gamma =0$). In this case, we notice that without generalized second-order information the cost is larger in all tests.

In Figure \ref{fig:ga} (A) the evolution of the cost is shown for different values of $\gamma$ and for $\beta=0.5$. The case $\gamma=0$ is excluded from the plot in view of its higher values, see Table \ref{t:1} below. 
\begin{table}[ht!]
\small
\begin{tabular}{l|l|l|l|l|l}
& $\beta =0.1$&  $\beta =0.3$ &  $\beta =0.5$&  $\beta =0.7$ &   $\beta =0.9$\\
\hline
$\gamma=0$& -2640.5471&-2095.9578&-1638.1323&-1258.8208&-946.8398\\\hline
$\gamma=50$& -2640.5586&-2096.427&-1639.4316&-1261.8356&-955.7\\\hline
$\gamma=500$&-2640.5623&\textbf{-2096.4514}&-1639.3464&\textbf{-1261.9043}&\textbf{-956.3495}\\\hline
$\gamma=1000$&-2640.5623&-2096.4502&\textbf{-1639.5237}&-1261.8978&-955.7178\\\hline
$\gamma=2000$&-2640.5623&-2096.252&-1639.521&-1261.5852&-952.9638\\\hline
$\gamma=5000$&-2640.5623&-2096.3521&-1639.515&-1261.3488&-956.0101\\\hline
$\gamma=8000$&-2640.5623&-2096.3442&-1639.527&-1261.2423&-954.4819\\\hline
$\gamma=10000$&\textbf{-2640.5625}&-2096.3564&-1639.4932&-1261.5385&-953.5706\\\hline
\end{tabular}
\caption{Cost function values varying parameters $\gamma$ and $\beta$}
\label{t:1}
\end{table}

Next, we test the importance of the active--set identification strategy described in Section \ref{sec: active identification}. We compare the computing time with respect to an implementation not hacking this strategy. We confirm the efficiency of using this strategy by measuring the computing time for this particular problem, see Table \ref{t:2}.

\begin{table}[ht!]
\small
\begin{tabular}{llllllll}
      & $\beta=0.35$ & $\beta=0.4$ & $\beta=0.45$ & $\beta=0.5$&$\beta=0.7$&$\beta=1$\\ \hline
	Active--set &  0.0018  & 0.0017 & 0.0017 & 0.0016 & 0.0018 & 0.0018 \\ \hline
	none        & 0.0031  & 0.0030 & 0.0031 & 0.0033&  0.0031&  0.0030
\end{tabular}
	\caption{Average time (in seconds) of the numerical solution of system \eqref{eq:direction} with and without the Active--Set strategy during the execution of GSOM algorithm.}
	\label{t:2}
\end{table}

\subsubsection{Numerical aspects of the second--order system} In this example, governed by the objective function {anisotrop} serves us to investigate the numerical efficiency regarding the numerical computation of the descent direction via system \eqref{eq:direction}. The structure of matrix $B^k + \beta \Gamma^k $ is determined by a particular problem and it is important to take it into account when it comes to choosing a linear solver or associated numerical strategies. In this particular example, the matrix $B+\beta\Gamma_k$ is a sparse banded matrix. Further, it is symmetric and positive definite; hence, we experiment with several iterative methods to observe the effect of the choice of the method solving the linear system.

 In Table \ref{t:method}, we compare \textit{direct methods} from Matlab's backslash, the \textit{preconditioned conjugate method} and the \textit{generalized minimum residual}. The last two preconditioned with the incomplete LU factorization (ilu), see \cite{saad2003}. We observe a substancial improvement using iterative methods which are suited for the structure of the matrix $B+\beta\Gamma_k$.
\begin{table}[ht!]
\small
\begin{tabular}{llllllll}
      & $m=1600$ & $m=2500$ & $m=3600$ \\ \hline
	direct & 0.028$\pm 3\times 10^{-6}$  & 0.0760$\pm 3\times10^{-5}$ & 0.180$\pm 1\times10^{-5}$\\ \hline 
	pcg &  0.005$\pm 5\times 10^{-8}$   & 0.0072$\pm 1\times10^{-7}$ & 0.011$\pm 4\times10^{-7}$  \\ \hline
	gmres & 0.017$\pm 1\times 10^{-6}$  & 0.0260$\pm 3\times10^{-6}$ & 0.056$\pm 9\times10^{-6}$ \end{tabular}
	\caption{Average$\pm$variance cpu--time (in seconds) for different linear solvers computing system \eqref{eq:direction} }
	\label{t:method}
\end{table}

\subsection{Image restoration}
Consider the image deconvolution example of \cite{gong2017}. The aim in this problem is to recover an image out of one convoluted with the random matrix $A$. For instance, this convolution occurs during the camera exposure, producing a blured image. If $x$ is the original image, the contaminated one is modeled by $y = Ax + z$, where $A \in \reals^{n\times n}$ and $z \in \reals^{n}$. The recovering process consists in choosing the image $x$ which best fits the observation and at the same time minimizes the term that computes the differences of each pixel with respect to its neighbors by means of a directed graph $G=\{N,E\}$. Thus, we look for a minimizer of the cost function
\begin{equation*}
f(x) = \frac12 \norm{A x - y}_2^2 + \alpha \norm{x}_1 + 
\beta \sum_{(i,j)\in E } |x_{i}-x_{j}|
\end{equation*}
Notice that the last function fits in our settings using the incidence matrix $C$, associated to the graph $G$, in order to express the penalizing term as: $ \sum_{(i,j)\in E } |x_{i}-x_{j}| = \norm{Cx}_1$.

In the following example we consider the recovering of an image of size $77\times 77$  from its corrupted observation $y=Ax+b$ with random noise $z$ with standard deviation $\sigma = 0.05$.  Here $A$ is a random (uniformly distributed) convolution matrix of size $2000 \times 5929$.

\begin{figure}[h]
\centering
\begin{subfigure}{0.45\textwidth}
\includegraphics{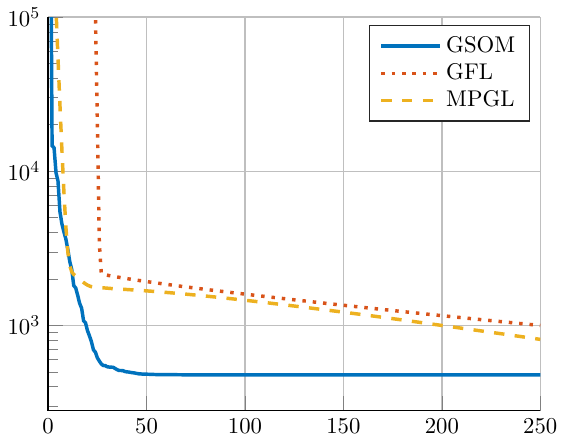}
%
\caption{Noise level: $\sigma=0.05$}
\end{subfigure}
\begin{subfigure}{0.45\textwidth}
\includegraphics{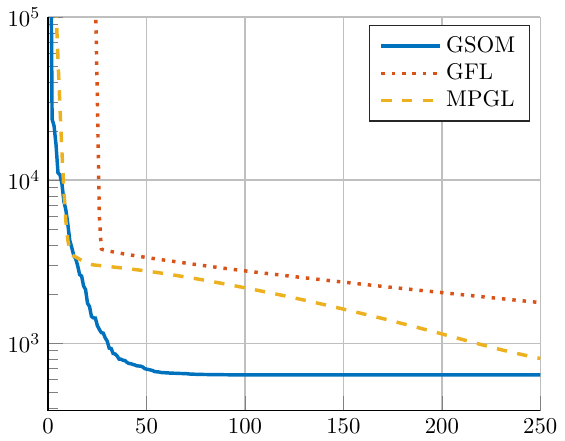}
%
\caption{Noise level: $\sigma=0.2$}
\end{subfigure}
\caption{History of the cost function for the 250 iterations}
\end{figure}
Next, we test the Cauchy--denoising model characterized by its non-Gaussian and impulsive property that preserves edges and details of images (see \cite{sciacchitano2015}). The anisotropic version of the discrete Cauchy denoising problem corresponds to the minimization of the nonconvex cost function:
\begin{equation}
\varphi(u)=\sum_{i} \log (a + (u_i - f_i)^2) + \beta \norm{C u}_1,
\end{equation}
where $C$ is the difference operator, $f$ is the observed image perturbed with Cauchy noise and $a>0$ is the scale parameter of the Cauchy distribution. Notice that the nonconvex structure of the optimization problem prevents the application of standard convex methods.

Again, an image of size $77 \times 77$ pixels is considered and a Cauchy--noise is added to the original image according to the formula
\begin{equation}
	f = u + v = u + \xi \frac{\eta_1}{\eta_2},
\end{equation}
suggested in \cite{sciacchitano2015}, where $\xi >0$ provides the noise level, and $\eta_i$, $i=1,2$, follow Gaussian distributions with mean 0 and variance 1. In the next experiment we chose $\xi= 0.01$.

\begin{figure}[h]
\centering
%
\begin{tabular}{cc}
Original & Cauchy--noised  \\
	\includegraphics{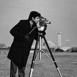} & \includegraphics{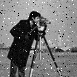}
	\end{tabular}
%
\end{figure}

The following set of pictures shows recovered images for different values of the scale parameter $a$ and the composite sparsity penalizing parameter $\beta$. Both play an important role in the restoration process. Indeed, we observe that larger values of $a$ result in a reduced level of  Cauchy-noise. The same observation applies to higher values of b. As usual, in this type of problems, there is a compromise between the amount of removed noise and the preservation of the details.
\begin{figure}[h]
\centering
%
\begin{tabular}{cccc}
&$\beta=0.5$ & $\beta=0.25$ & $\beta=0.1$  \\
$a=0.9$&	\includegraphics{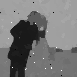} & \includegraphics{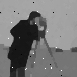}& \includegraphics{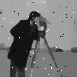}\\
$a=0.6$&	\includegraphics{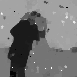} & \includegraphics{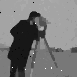}& \includegraphics{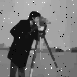}\\
$a=0.3$&	\includegraphics{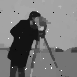} & \includegraphics{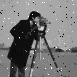}& \includegraphics{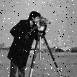}\\
	\end{tabular}
%
\caption{Recovered images by GSOM method}
\end{figure}

Because of the nonconvexity of the Cauchy problem, standard first-order methods cannot be applied. There exist methods designed for nonconvex problems; for instance, the iPiano algorithm, see \cite{ochs2014ipiano}, which is based on a forward-backward splitting with inertial splitting techniques. In each step, iPiano requires the computation of the proximal mapping:
\begin{equation}
\hat x\mapsto \argmin_{x\in\reals^n} \l\{ \frac12\norm{x-\hat x}_2^2 +\alpha \norm{Cx}_1\r\}	 \label{eq:prox_ip}
\end{equation}
which falls in the convex case of \eqref{eq:P}. Consequently, previous methods used in the experiments may be applied for evaluating \eqref{eq:prox_ip}. 

One of the cavils of second--order methods is the memory limitation related to the storage of the matrix of the second-order system \label{eq:direction}. In particular, for image processing and, in general, for applications which involve huge amounts of data, the numerical solution of this system can be prohibitive.  However,  there are a lot of techniques that can be utilized to overcome such inconvenience. 

Aiming to illustrate a practical utilization of such techniques, we give a glimpse of parallel preconditioning. Taking into account the matrix structure for the Cauchy problem, we apply a Bock--Jacoby type preconditioning (\cite[Sec.12.2]{saad2003}) by dividing the system into smaller systems that are solved separately and then gathering each overlapping portion of the solution into a single one. Observe in this case that the matrix $B^k + \beta \Gamma^k$ has a banded sparse structure. We explain the numerical scheme subdividing in two subproblems, but it can be easily extended for an arbitrary number of $p$ partitions.  Assuming $m$ an even integer, and let $l$ be the number of overlaping entries,  we define $A_1$ and $A_2$ by choosing  $a_{i,j}$ for $i,j=1,\ldots,\frac{m}{2}+l$ as the entries of $A_1$ and $a_{i,j}$ for $i,j=\frac{m}{2}+1,\ldots,m$ as the entries of $A_2$. 

The updating scheme is given by
\begin{equation}
x^{k+1} =x^{k}	+ V_1A^{-1}V^T_1 r^k + V_2A_2^{-1}V_2 r^k,
\end{equation}
 where $V_i$ are subspace projections, $A_i$ are block diagonal overlaping submatrices of $A = B^k + \beta \Gamma^k$, and $r^k$ is the residual; i.e. $r^k = Ax^k + [{\nabla}f(x^k) + \beta C^\top\xi(x^k)]$. The inverse--vector multiplication operations are performed using  direct or iterative methods. In our example, for reference we use direct methods of Matlab. In the next table, we can realize how this partition reduces the memory cost, specifically for the system matrix for the Cauchy problem of a picture of $103\times 103$ pixels using 20\% of the partitions as overlapping size. As shown in Figure \ref{fig:schwarz1}, solving the partitioned system additively takes slightly longer time than solving the full system at once. However, it is a low price to pay if memory storage utilization needs to be reduced drastically. We observe this effect in Table \ref{t:additive}.
\begin{table}[h!]
\begin{tabular}{llllllll}
       Partitions & Size of $A_i$ (max) & Storage (Kb) \\ \hline
	- & 10609$\times$10609  &424 \\ \hline 
	$p= 3$ & 4242 $\times$ 4242   & 150 &  \\ \hline
	$p=4$ & 3181$\times$3181 & 84
\end{tabular}
	\caption{Computing time solving system \eqref{eq:direction} using Block--Jacobi}
	\label{t:additive}
\end{table}

\begin{figure}[ht!]
\centering
\includegraphics{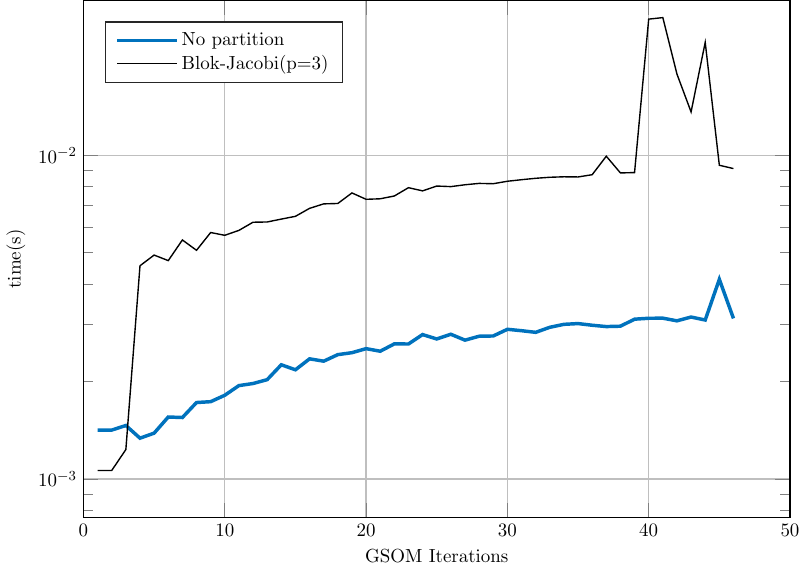}
\caption{Time in each iteration for the computation of $d$}
\label{fig:schwarz1}
\end{figure}

\subsection{Graph trend filtering} In \cite{wang2016trend} the authors introduced a technique of filtering data over graphs, that was applied in the denoising over graphs using the discrete laplacian as sparsity--inducting operator. There, it was showed that better results may be achieved compared to other denoising thechniques. In our setting, $C = \Delta^{(2)}$, where for a integer $k$ the operator $\Delta^{(k)}$ is defined recursively by
\begin{equation}
	\Delta^{(k+1)}:=
	\begin{cases}
		(\Delta^{(1)})^\top \Delta^{(k)}, & \text{ if } k \text{ is odd}, \\
		\Delta^{(1)} \Delta^{(k)}, & \text{ if } k \text{ is even},
	\end{cases}
\end{equation}
where $\Delta^{(1)}$ is the oriented incidence matrix of the graph. Notice that $\norm{\Delta^{(1)}x}_1=\sum_{(i,j)\in E} |x_i-x_j|$, where we denote the graph $G=\{N,E\}$. Therefore, $\Delta^{(2)}={\Delta^{(1)}}^\top \Delta^{(1)}$.

As an example, we consider the denoising of COVID--19 data over a graph corresponding to the Pichincha province of Ecuador connecting adjacent areas or tracts. Hence, each node corresponds to a particular tract of the province territory. The signal data considered in each node consist of the reported number of cases of each tract, denoted by $y$. The noise in this kind of data comes from an imprecise assignments within tracts, counting errors, false positive or negative cases, among other. In our example, we assume that the noise induced by these different sources is normally distributed $y \sim N(x_0,\sigma^2I)$. The sparse graph filtering problem aims to minimize the following cost
\begin{equation}
	f(x) = \frac12 \norm{x-y}_2^2 + \beta_1 \norm{\Delta^{(2)}x}_1 + \beta_2  \norm{x}_1
\end{equation}
Figure \ref{fig:gft1} shows an expected behavior of a first order method (ADMM) compared with a second--order method (GSOM). We observed that GSOM is faster and more precise. However, it requires the solution of a linear system, which may be costly. Nevertheless,  the computational cost can be outstripped by utilizing parallelization and numerical techniques.
\begin{figure}[ht!]
\centering
\includegraphics{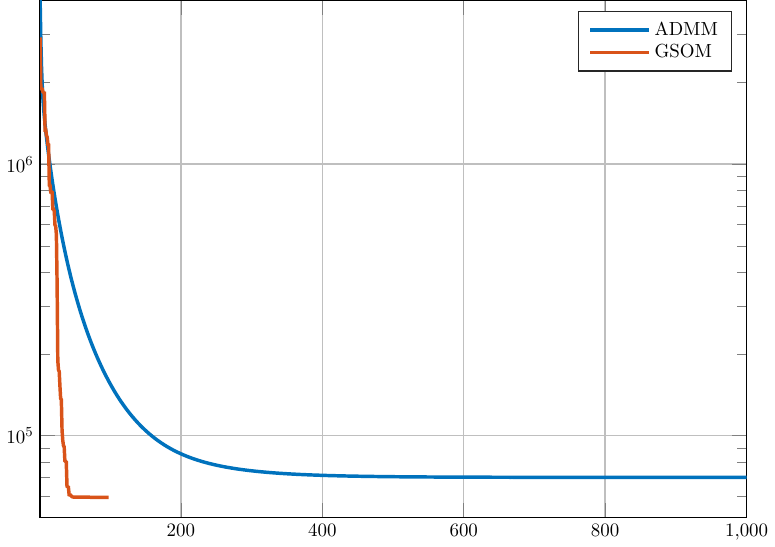}
\caption{Comparison with Fast ADMM algorithm \cite{wang2016trend}}
\label{fig:gft1}
\end{figure}
\begin{figure}[ht!]
\begin{subfigure}{.49\textwidth}
  \includegraphics[width=\textwidth]{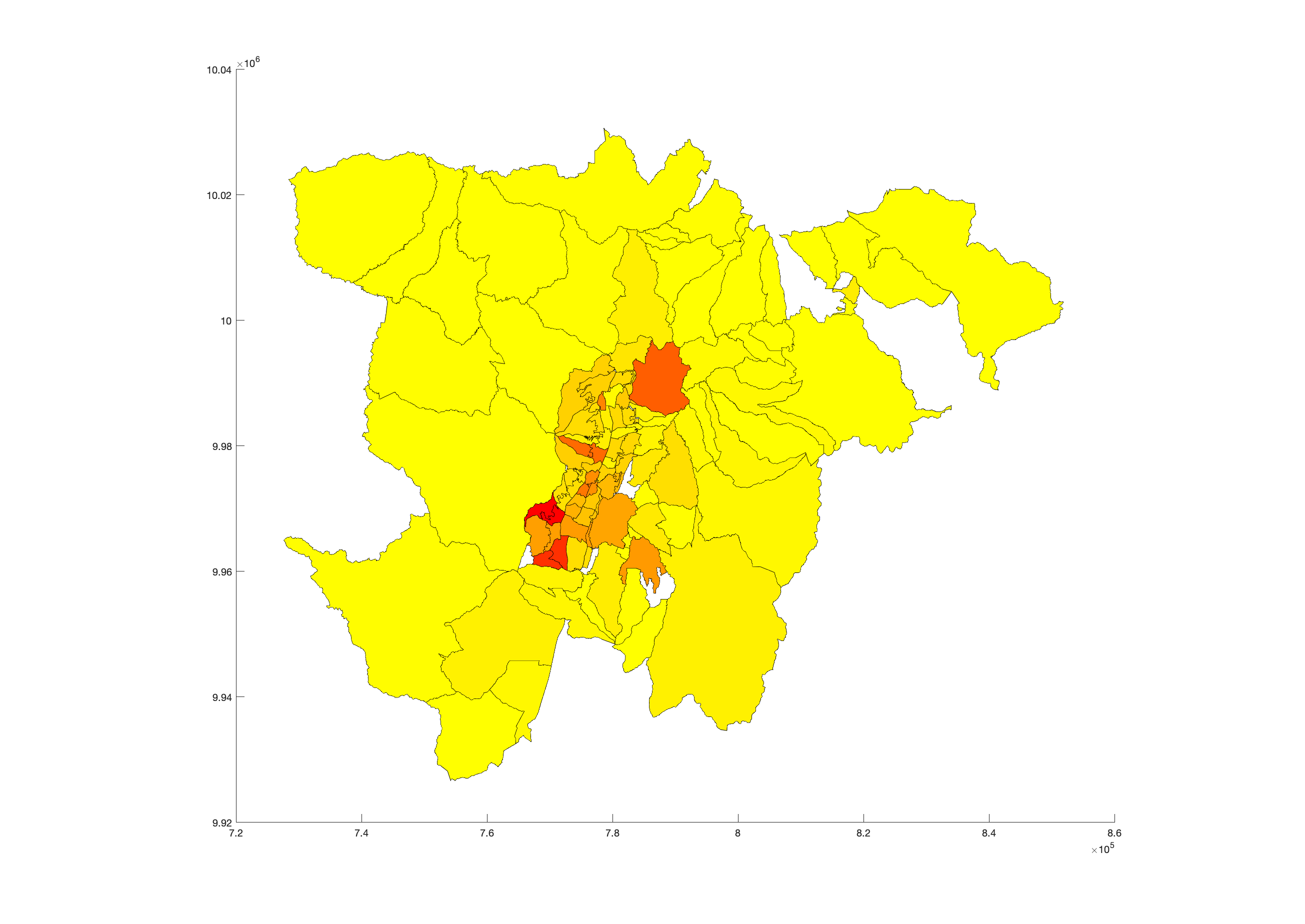}
  \caption{Original graph data}
  \end{subfigure}
  \begin{subfigure}{.49\textwidth}
  \includegraphics[width=\textwidth]{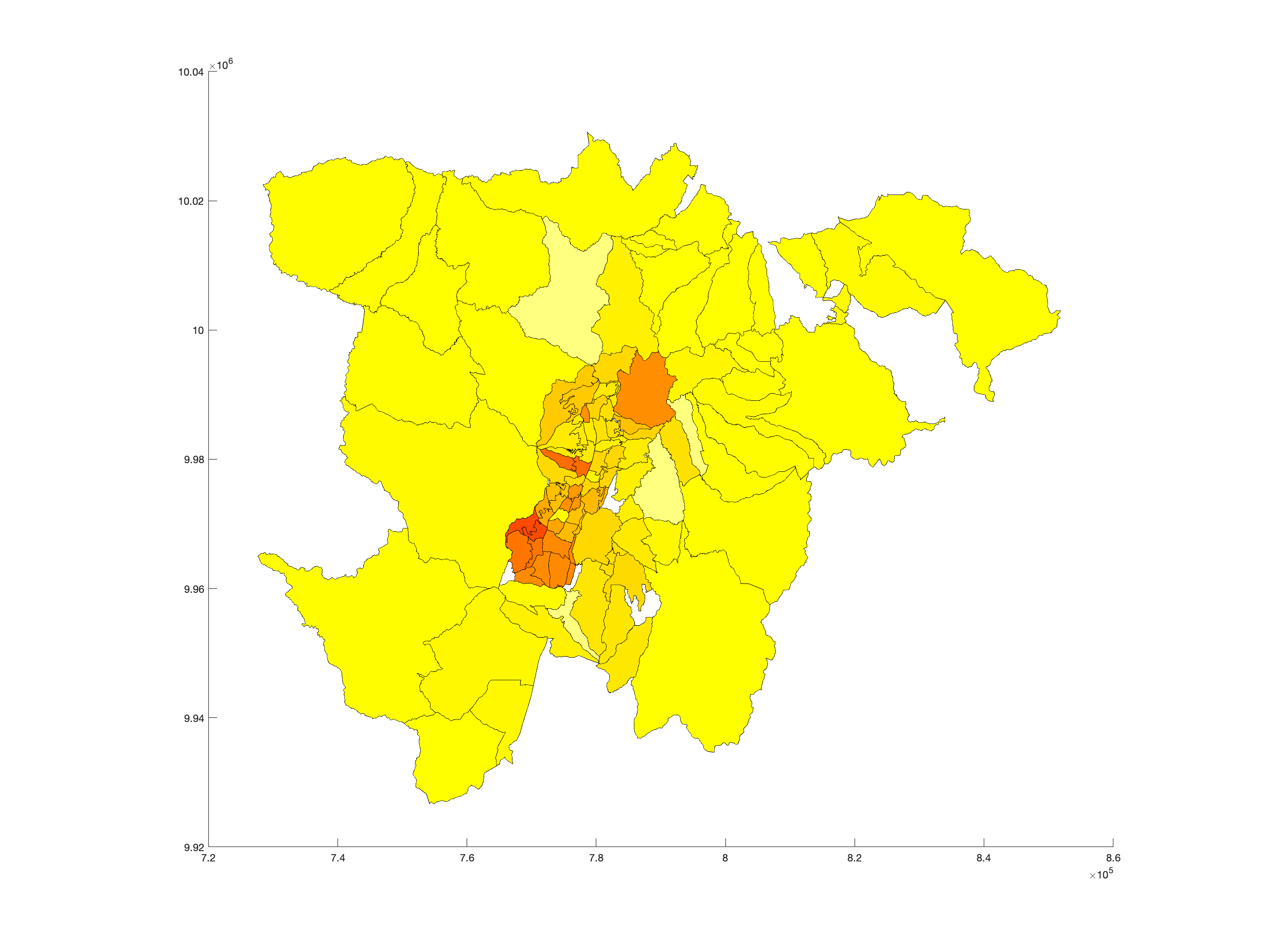}
  \caption{Filtered graph data}
  \end{subfigure}
  \caption{Graph trend filtering of COVID-19 data in a graph of Pichincha--Ecuador}
\end{figure}

\bibliographystyle{plain}

\end{document}